\documentclass{amsart}

\usepackage{amsmath}
\usepackage{amssymb}
\usepackage{tikz-cd} 
\usepackage{amscd, mathtools}
\usepackage[colorlinks,pagebackref=true]{hyperref}
\usepackage[all]{xy}

\usepackage{enumerate}

\numberwithin{equation}{section}

\def\Ext{\operatorname{Ext}}

\def\ann{\operatorname{ann}}

\def\cm{\operatorname{CM}}

\def\depth{\operatorname{depth}}

\def\Ext{\operatorname{Ext}}

\def\Hom{\operatorname{Hom}}

\def\m{\mathfrak{m}}

\def\mod{\operatorname{mod}}

\def\p{\mathfrak{p}}

\def\spec{\operatorname{Spec}}

\def\syz{\mathrm{\Omega}}

\def\tr{\operatorname{tr}}
\def\@citecolor{blue}
\def\@linkcolor{blue}
\def\@urlcolor{blue}
\def\@urlcolor{blue}
\def\depth{\operatorname{depth}}

\def\Hom{\operatorname{Hom}}

\def\sup{\operatorname{sup}}
\def\uac1{\operatorname{uac1}}
\def \refl {\operatorname{Ref}}  
\theoremstyle{plain} %Shinya added
\newtheorem{Theorem}{Theorem}[section]

\newtheorem{thm}[Theorem]{Theorem} %Shinya added
\newtheorem{lem}[Theorem]{Lemma}%Shinya added
\newtheorem{prop}[Theorem]{Proposition}%Shinya added
\newtheorem{cor}[Theorem]{Corollary}%Shinya added

\theoremstyle{definition} %Shinya added

\newtheorem{chunk}[Theorem]{\hspace*{-1.065ex}\bf}

\newtheorem{rem}[Theorem]{Remark}

\let\epsilon\varepsilon
\let\phi=\varphi
\let\kappa=\varkappa

\def \c {\mathfrak c}
\def \grade {\operatorname{grade}}

\def \Ass {\operatorname{Ass}}  
\def \Syz {\operatorname {Syz}}

\title{finite birational extension with stable conductor}
\author{Souvik Dey}  
\address{Souvik Dey: Department of Mathematics \\ University of Kansas\\405 Snow Hall, 1460 Jayhawk Blvd.\\ Lawrence, KS 66045, U.S.A.}
\email{souvik@ku.edu}  

\begin{document}

\begin{abstract} Let $S$ be a module finite birational extension of a $1$-dimensional local Cohen--Macaulay ring $R$. When is the conductor of $S$ in $R$ a stable ideal? If $R$ is also generically Gorenstein, then we show that the conductor of $S$ in $R$ is a stable ideal, and $S$ is a reflexive $R$-module if and only if $\syz \cm(S)=\cm(S)\cap \syz \cm(R)$.             
\end{abstract}  

\maketitle

\section{Introduction} Let $I$ be an ideal of a commutative Noetherian ring $R$. $I$ is called a regular ideal if $I$ contains a non-zero-divisor of $R$. A regular ideal $I$ is stable $I\cong \Hom_R(I,I)$ holds, or equivalently, if $I\cong (I:_{Q(R)}I)$ holds \cite[Section 3, Proposition 3.2]{dl}, where $Q(R)$ is the total ring of fractions of $R$.      

Given a $1$-dimensional local Cohen--Macaulay ring $R$  whose integral closure $\overline R$ (in the total ring of fractions $Q(R)$ of $R$)  is module finite over $R$, it is well-known that the conductor of $\overline R$, i.e., $\c=(R:_{Q(R)} \overline  R)$  is a regular stable ideal of $R$. Motivated by this, we raise the question: Let $S$ be a module finite extension of a $1$-dimensional local Cohen--Macaulay ring $R$ such that $R\subseteq S\subseteq Q(R)$. When is the conductor $\c_R(S):=(R:_{Q(R)} S)$ of $S$ in $R$ a stable ideal of $R$?  

When $R$ is generically Gorenstein, and $S$ is moreover a reflexive $R$-module, we characterize such finite birational extensions in terms of a special description of their category of maximal Cohen--Macaulay modules and reflexive modules. Namely, we show

\begin{thm}\label{main1} Let $R$ be a generically Gorenstein local Cohen--Macaulay ring of dimension $1$ with total ring of fractions  $Q(R)$. Let $R\subseteq S\subseteq Q(R)$ be a module finite birational extension. Then, the following are equivalent  

\begin{enumerate}[\rm(1)]
\item $S \in \refl(R)$, and the conductor $\c_R(S)=(R:_{Q(R)} S)$ is a stable ideal.    

\item $S\cong \Hom_R(S,R)$.   

\item $\refl(S)=\cm(S)\cap \refl(R)$.     

\end{enumerate}

\end{thm} 

Here, a commutative Noetherian ring is called generically Gorenstein if its localizations at all associated prime ideals are Gorenstein, and for a ring $A$, $\cm(A), \refl(A)$ denote the categories of maximal Cohen--Macaulay and reflexive (finitely generated) $A$-modules respectively.    

Theorem \ref{main1} is proved in Section \ref{ms} (Theorem \ref{main11}). For this, some preliminary notions and results are recalled and proved in Section \ref{ps}.  

\section{preliminaries}\label{ps} 

Throughout, all rings are commutative and Noetherian. All modules are assumed to be finitely generated. 

Given a ring $A$ with total ring of fractions $Q(A)$, and $A$-submodules $M,N$ of $Q(A)$, we write $(N:M)$ to denote the $A$-submodule $\{a\in Q(A): aM\subseteq N\}$ of $Q(A)$. An $A$-submodule $M$ of $Q(A)$ is called \textit{regular} if it contains a non-zero-divisor of $A$. Given $A$-submodules $M,N$ of $Q(A)$, with $M$ regular, $(N:M)$ can be naturally identified with $\Hom_A(M,N)$, see \cite[Propostion 2.4(1)]{tk}. We call a ring extension $A\subseteq B$ \textit{birationa} if $B\subseteq Q(A)$,i.e, $Q(A)=Q(B)$. A birational extension $A\subseteq B$ is called a \textit{finite birational} extension if $B$ is module finite over $A$. If $A\subseteq B$ is a birational extension and $M,N$ are $B$-modules such that $N$ is torsion-free, then $\Hom_A(M,N)=\Hom_B(M,N)$ (see \cite[4.14(i)]{lw} and its proof).    

It is straightforward to verify that given a birational extension $A\subseteq B$, the conductor $\c_A(B)=(A:B)$ ideal is a regular ideal if and only if $B$ is module finite over $A$. 
It is clear that $\c_A(B)$ is an ideal of both $B$ and $A$. Since $\c_A(B)$ is an ideal, it is torsion-free, hence $\Hom_A(\c_A(B),\c_A(B))=\Hom_B(\c_A(B),\c_A(B))$. Thus, $\c_A(B)$ is a stable ideal of $A$ if and only if it is a stable ideal of $B$.          

\begin{chunk}\label{syzcm} We denote the category of $A$-modules, maximal Cohen--Macaulay (MCM for short) $A$-modules, and reflexive $A$-modules by $\mod A$, $\cm(A)$, and $\refl(A)$ respectively.    
Given a subcategory $\mathcal X$ of $\mod A$, and integer $n\ge 1$, we denote by $\syz^n(\mathcal X)$ the subcategory of all $A$-modules $M$ for which there exists an exact sequence $0\to M \to P_{n-1}\to \cdots \to P_0\to N \to 0$, where each $P_i$ is a projective $A$-module and $N\in \mathcal X$. We denote $\syz^n(\mod A)$ by just $\Syz_n(A)$ as well. If $M\in \refl(A)$, then taking a free presentation of $\Hom_A(M,A)$ and dualizing by $A$, it is clear that there exists an exact sequence of $A$-modules $0\to M \to A^{\oplus a} \to A^{\oplus b}$ for some integers $a,b \ge 1$. Hence, $\refl(A)\subseteq \Syz_2(A)$. 

If $M\in \syz \cm(A)$, then there exists an exact sequence of $A$-modules $0\to M \to P \to N \to 0$ where $N\in \cm(A)$ and $P$ is a projective $A$-module. Now $P\oplus P'\cong A^{\oplus n}$ for some projective  $A$-module $P'$ (hence $P'\in \cm(A)$). Hence, we get an exact sequence $0\to M \to A^{\oplus n} \to N \oplus P' \to 0$, where $N\oplus P' \in \cm(A)$.  
\end{chunk}

\begin{chunk}\label{refr} Let $M$ be an $A$-submodule of $Q(A)$ containing a non-zero-divisor. Then, $M$ is reflexive if and only if $A:(A:M)=M$, see \cite[Proposition 2.4(4)]{tk}. 

\end{chunk}

\if0
\begin{prop}  Let $R\subseteq S$ be a module finite birational extension. Then, the following are equivalent     

\begin{enumerate}[\rm(1)]
\item $S\in \syz \cm(R)$. 

\item $\syz \cm(S)\subseteq \syz \cm(R)$.    
\end{enumerate}
\end{prop}

\begin{proof} $(2)\implies (1)$: $S\in \syz \cm(S)\subseteq \syz \cm(R)$. 

$(1)\implies (2)$: Let $M\in \syz \cm(S)$. Then, we have an exact sequence of $S$-modules (hence also of $R$-modules) $0\to M \to S^{\oplus a} \to N\to 0$, where $N\in \cm(S)\subseteq \cm(R)$. 

\end{proof}
\fi

\begin{chunk}\label{mat} Let $A$ be a generically Gorenstein $1$-dimensional Cohen--Macaulay ring. Then, by \cite[Theorem 2.3$(2)\iff (7)$]{matsui} we have $\cm(A)=\Syz_1(A)$. Hence, $\syz \cm(A)=\Syz_2(A)=\refl(A)$, where the last equality is by \cite[Theorem 2.3$(2)\iff (6)$]{matsui}.  
\end{chunk}

If $R$ is a $1$-dimensional Cohen--Macaulay ring and $S$ is a finite birational extension of $R$, then $S$ is also $1$-dimensional and Cohen--Macaulay and so consequently, $\cm(S)=\text{torsion-free } S \text{-modules} \subseteq \text{torsion-free } R \text{-modules}=\cm(R)$. Moreover, $\prod_{\p \in \Ass(R)} R_{\p} \cong Q(R)=Q(S)\cong \prod_{Q\in \Ass(S)} S_{Q}$; hence if $R$ is generically Gorenstein, then so is $S$.

For a module $M$ over a ring $R$, we denote by $\tr_R(M)$ the trace ideal of $M$ (\cite{lin}, \cite[Definition 2.2]{tk}). If $M$ is an $R$-submodule of $Q(R)$ containing a non-zero-divisor, then $\tr_R(M)=(R:M)M$, see \cite[Proposition 2.4(2)]{tk}.    

\begin{chunk}\label{trco} Let $R\subseteq S$ be a finite birational extension of rings. Then, it follows from \cite[Theorem 2.9]{dms} that $\cm(S)\cap \refl(R)=\{M\in \refl(R): \tr_{R}(M)\subseteq \c_R(S)\}$.       
\end{chunk}     

\begin{prop}\label{pre} Let $R$ be a generically Gorenstein $1$-dimensional Cohen--Macaulay ring, and let $R\subseteq S$ be a module finite birational extension. Then, the following are equivalent     

\begin{enumerate}[\rm(1)]
\item $S\in \refl(R)$. 

\item $\syz \cm(S)\subseteq \syz \cm(R)$.

\item $\refl(S)\subseteq \refl(R)$. 

\item  $\refl(S)\subseteq \{M\in \refl(R): \tr_{R}(M)\subseteq \c_R(S)\}$.  
\end{enumerate}
\end{prop}  

\begin{proof} $(2)\implies (1)$: Follows from $S\in \syz \cm(S)\subseteq \syz \cm(R)=\refl(R)$ (see \ref{mat}).    

$(1)\implies (2)$: Let $M\in \syz \cm(S)$. Then, by the discussion at the end of \ref{syzcm}, we have an exact sequence of $S$-modules (hence also of $R$-modules) $0\to M \to S^{\oplus a} \to N\to 0$, where $N\in \cm(S)\subseteq \cm(R)$. As $S\in \refl(R) =\syz \cm(R)$, so $0\to S \to R^{\oplus b} \to X \to 0$ for some $X\in \cm(R)$. So we have the following push-out diagram of $R$-modules

\begin{tikzcd}
            &                       & 0 \arrow[d]                       &                       &   \\
0 \arrow[r] &  M \arrow[r]          &  S^{\oplus a} \arrow[r] \arrow[d] & N \arrow[d] \arrow[r] & 0 \\
0 \arrow[r] & M \arrow[r] \arrow[u,equal] & R^{\oplus ab} \arrow[d] \arrow[r] & Y \arrow[r] \arrow[d] & 0 \\
            &                       & X^{\oplus a} \arrow[d]            & X^{\oplus a} \arrow[d] \arrow[l,equal] &   \\
            &                       & 0                                 & 0                     &  
\end{tikzcd}

Since $N,X^{\oplus a}\in \cm(R)$, so $Y\in \cm(R)$, so $M\in \syz \cm(R)$.  

$(2)\iff (3)$: Follows from \ref{mat}.  

$(3)\iff (4)$: Follows from  $\refl(S)\subseteq \cm(S)$ (as $\dim S=1$)  and \ref{trco}.     
\end{proof}

We notice one last general observation that we will use in the next section 

\begin{chunk}\label{sim} Let $I$ be an ideal of a ring $R$ such that $R$ is a direct summand of $I$. Then, $I\cong R$. Indeed, by assumption, there exists an $R$-module $N$ and an $R$-linear isomorphism $f:R \oplus N\to I$. it is enough to show that $N=0$. So let $x\in N$. Let $r:=f(0,x)$ and $s:=f(1,0)$. Then, $r,s\in R$, so $f(0,sx)=sf(0,x)=sr=rs=rf(1,0)=f(r,0)$. As $f$ is an isomorphism, we get $(0,sx)=(r,0)$. Hence $r=0$ i.e. $f(0,x)=0$. Hence $f$ is an isomorphism implies $x=0$. As $x\in N$ was arbitrary, we are done.     
\end{chunk}   

\section{main result}\label{ms}
To prove the main result Theorem \ref{main1}, we need some preparatory results regarding trace ideals. First, we record a general lemma 

\begin{lem}\label{2} Let $I$ be an ideal of a ring $A$  containing a non-zero-divisor of $A$. Then, $\Ext^1_A(A/I,A)\cong \dfrac{(A:I)}{A}$.  Consequently, $\Ext^1_A(A/I,A)=0$ if and only if $(A:I)=A$.    
\end{lem} 

\begin{proof} Consider the natural exact sequence $0\to I \xrightarrow{i} A \to A/I \to 0$. Applying $\Hom_A(-,A)$  we get exact sequence $0\to \Hom_A(A/I,A)\to \Hom_A(A,A) \xrightarrow{i^*} \Hom_A(I,A)\to \Ext^1_A(A/I,A)\to \Ext^1_A(A,A)=0$. Now, $\Hom_A(A/I,I)$ is killed by $I$ hence killed by a non-zero-divisor, so $\Hom_A(A/I,A)$ is a torsion $A$-module. But also, $\Hom_A(A/I,A)$ embeds in $\Hom_A(A,A)\cong A$, hence $\Hom_A(A/I,A)$ is also a torsion-free $A$-module. Thus $\Hom_A(A/I,J)=0$. Hence we get exact sequence    

 $$0\to  \Hom_A(A,A) \xrightarrow{i^*} \Hom_A(I,A)\to \Ext^1_A(A/I,A)\to 0$$  

Also, we have the following commutative diagram 

$$
\begin{tikzcd}
{\text{Hom}_A(A,A)} \arrow[r, "f\mapsto f|_I"]     & {\text{Hom}_A(I,A)}                     \\
A \arrow[u, "y\to\{z\mapsto yz\}"] \arrow[r, "i"'] & (A:I) \arrow[u, "x\to\{m\mapsto xm\}"']
\end{tikzcd}$$     

where the vertical arrows are isomorphisms and the lower row arrow is the natural inclusion. Hence we get the following exact sequence $0\to A \xrightarrow{i} (A:I) \to \Ext^1_A(A/I,A)\to 0$, and the isomorphism follows.  Consequently, $\Ext^1_A(A/I,A)=0$ if and only if $A=(A:I)$.  
\end{proof}

For an ideal $I$ of a ring $A$, $\grade_A(I)$ denotes the length of a maximal $A$-regular sequence of $A$. It is well-known (see \cite{bh}) that $\grade_A(I)=\inf\{i:\Ext^i_A(A/I,A)\ne 0\}=\inf\{\depth A_{\p}: \p\in V(I)\}\le \sup\{\depth A_{\m} : \m \in \text{Max} \spec(A)\}$.    

\begin{prop}\label{trmain} Let $I$ be a reflexive ideal of a ring $A$. Then, the following are equivalent

\begin{enumerate}[\rm(1)]
    \item $\grade_A(\tr_A(I))\ge 2$.
    
    \item $I$ contains a non-zero-divisor and $(I:I)=A$.  
\end{enumerate}

\end{prop}  

\begin{proof} $(2)\implies (1)$: As $I$ is reflexive, so $I=(A:(A:I))$ (see \ref{refr}). Now, $$A=(I:I)=\left((A:(A:I)):I\right)=(A:(A:I)I)=(R:\tr_A(I))$$  Hence, $\Ext^1_A(A/\tr_A(I),R)=0$, by Lemma \ref{2}. Thus, $\grade_A (\tr_A(I))\ge 2$.  

$(1)\implies (2)$: Since $\grade_A(\tr_A(I))\ge 2$, so $\tr_A(I)$ contains a non-zero-divisor. First we show that this implies $I$ contains a non-zero-divisor. Indeed, by definition, we have a surjection $I\otimes_A \Hom_A(I,A)\to \tr_A(I)\to 0$. Since $I$ is finitely generated, we have a surjection $I^{\oplus n} \to \Hom_A(I,A)\to 0$. Hence, we have a surjection $I^{\oplus n }\to \tr_A(I)\to 0$. Hence $\ann_A(I)\subseteq \ann_A(\tr_A(I))=0$, where the later is $0$ since $\tr_A(I)$ contains a non-zero-divisor. If $I$ consisted of zero-divisors, then by prime avoidance $I\subseteq \p$ for some $\p \in \Ass(A)$. As $\p=\ann_A(x)$ for some $0\ne x\in A$, we would then have $xI=0$, i.e., $x\in \ann_A(I)=0$, contradiction! Thus, $I$ contains a non-zero-divisor. Hence, $\tr_A(I)=(A:I)I$ by \cite[Proposition 2.4(2)]{tk}. Moreover, $I$ is reflexive, so $I=(A:(A:I))$ (\ref{refr}). As $\grade_A(\tr_A(I))\ge 2$, so $\Ext^1_A(A/\tr_A(I),A)=0$, so by Lemma \ref{2}  we get $$A=(A:\tr_A(I))=(A:(A:I)I)=\left((A:(A:I)):I\right)=(I:I)$$      
\end{proof}

\begin{cor}\label{maincor} Let $A$ be a semi-local ring such that $\sup\{\depth A_{\m} : \m \in \operatorname{Max} \spec(A)\}\le 1$. Let $I$ be a reflexive ideal containing a non-zero-divisor such that $(I:I)=A$. Then, $I\cong A$.  
\end{cor}    

\begin{proof}  By Proposition \ref{trmain} we get $\grade_A(\tr_A(I))\ge 2$. If $\tr_A(I)\neq A$, then there exists a maximal ideal $\m \in V(\tr_A(I))$, and then $\grade_A(\tr_A(I))=\inf\{\depth A_{\p}:\p\in V(\tr_A(I))\}\le \depth A_{\m}\le 1$, contradiction! Thus, $\tr_A(I)=A$. Localizing at each prime ideal we get that $\tr_{A_{\p}}(I_{\p})=A_{\p}$ for every prime ideal $\p$ of $A$.  By \cite[Proposition 2.8(iii)]{lin} we see that $A_{\p}$ is a direct summand of $I_{\p}$ for every prime ideal $\p$ of $A$. Then by \ref{sim} we get that $I_{\p}\cong A_{\p}$ for every prime ideal $\p$ of $A$. Thus, $I$ is a projective module of constant rank over the semi-local ring $A$, hence by \cite[Lemma 1.4.4]{bh} we get $I\cong A$.     
\end{proof}

In view of \ref{mat}, Theorem \ref{main1} follows from the following result which is our main theorem   

\begin{thm}\label{main11} Let $R$ be a generically Gorenstein local Cohen--Macaulay ring of dimension $1$ with total ring of fractions  $Q(R)$. Let $R\subseteq S\subseteq Q(R)$ be a module finite birational extension. Then, the following are equivalent  

\begin{enumerate}[\rm(1)]

\item $\syz \cm(S)=\cm(S)\cap \syz \cm(R)$. 

\item $S\cong \Hom_R(S,R)$. 

\item $S\in \refl(R)$, and $\c_R(S)$ is a stable ideal.  

\item $\syz \cm(S)=\{M\in \refl(R): \tr_{R}(M)\subseteq \c_R(S)\}$.   
\end{enumerate}

\end{thm}  

\begin{proof}  First we notice that as $R$ is local, so $S$ is semi-local. Moreover, $S$ is also $1$-dimensional, Cohen--Macaulay, and generically Gorenstein.   

$(1)\implies (2)$: Firstly, the hypothesis of (1) implies $S$ is a reflexive $R$-module by Proposition\ref{pre}. Also, we know $\c_R(S)=(R:S)\cong \Hom_R(S,R)\in \Syz_2(R)$. By \ref{mat}, $\c_R(S)\in \syz \cm(R)$. Moreover, $\c_R(S)$ is an ideal of $S$. Thus, $\c_R(S)\in \cm(S)\cap \syz \cm(R)$. Thus, $\c_R(S)\in \syz \cm(S)=\refl(S)$ by hypothesis  and \ref{mat}. Moreover, keeping in mind $Q(S)=Q(R)$, we see $$(\c_R(S):_{Q(S)}\c_R(S))=((R:S):\c_R(S))=(R:S\c_R(S))=(R:\c_R(S))=(R:(R:S))=S$$ where we used that $\c_R(S)$ is an ideal of $S$ and that $S$ is a reflexive $R$-module. Thus by Corollary \ref{maincor} we get $S\cong \c_R(S)$.   

$(2)\implies (3)$: As $S^*\in \Syz_2(R)=\refl(R)$ by \ref{mat}, we readily get $S$ is a reflexive $R$-module from the hypothesis of (2). Moreover, $$(\c_R(S):_{Q(S)}\c_R(S))=((R:S):\c_R(S))=(R:S\c_R(S))=(R:\c_R(S))=(R:(R:S))=S\cong \Hom_R(S,R)$$

Keeping in mind $\c_R(S)\cong \Hom_R(S,R)$ and $(\c_R(S):\c_R(S))\cong \Hom_R(\c_R(S),\c_R(S))$ (as $\c_R(S)$ contains a non-zero-divisor of $R$ since $S$ is module finite over $R$), we see that $\c_R(S)\cong \Hom_R(\c_R(S),\c_R(S))$, hence $\c_R(S)$ is a stable ideal.     

$(3)\implies (2)$: Since $\c_R(S)$ is a stable ideal, so $\c_R(S)\cong (\c_R(S):\c_R(S))$. But also, $S$ is reflexive gives  $$(\c_R(S):_{Q(S)}\c_R(S))=((R:S):\c_R(S))=(R:S\c_R(S))=(R:\c_R(S))=(R:(R:S))=S$$

Thus, $S\cong \c_R(S)\cong \Hom_R(S,R)$.  

$(2)\implies (1)$:  As $S$ and $\Hom_R(S,R)\cong \c_R(S)$ are torsion-free $S$-modules, so any $R$-linear map between them is also $S$-linear, hence, the isomorphism in hypothesis of (2) is also $S$-linear. Since $S \cong \Hom_R(S,R)$ is a reflexive $R$-module, so $\syz \cm(S)\subseteq \cm(S) \cap \syz \cm(R)$ by Proposition \ref{pre}.  For the reverse inclusion, let $M\in \cm(S) \cap \syz \cm(R)$. Then, $\Hom_R(M,R)\in \cm(S)$.  Choose an $S$-module presentation $S^{\oplus a} \to S^{\oplus b} \to \Hom_R(M,R) \to 0$. Applying $\Hom_R(-,R)$ and remembering  $M\in \syz \cm(R)=\refl(R)$ and $S\cong \Hom_R(S,R)$, we get an exact sequence of $R$-modules $0\to M \to S^{\oplus a} \to S^{\oplus b}$. Since $M, S\in \cm(S)$ are torsion-free, so any $R$-linear map between them is $S$-linear, thus $M\in \Syz_2(S)=\syz \cm(S)$, where the last equality is by \ref{mat}. This proves the reverse inclusion $\cm(S) \cap \syz \cm(R)\subseteq \syz \cm(S)$.  

$(1)\iff (4)$: Follows by \ref{trco}.   

\end{proof}

\begin{cor}\label{corm}  Let $R$ be a generically Gorenstein local Cohen--Macaulay ring of dimension $1$ with total ring of fractions  $Q(R)$. Let $R\subseteq S\subseteq Q(R)$ be a module finite birational extension. Assume $S$ is a Gorenstein ring. Then, the following are equivalent  

\begin{enumerate}[\rm(1)]

\item $\cm(S)\subseteq \refl(R)$.    

\item $S\cong \Hom_R(S,R)$. 

\item $S\in \refl(R)$, and $\c_R(S)$ is a stable ideal.  

\item $\cm(S)=\{M\in \refl(R): \tr_{R}(M)\subseteq \c_R(S)\}$.   
\end{enumerate}   
\end{cor}

\begin{proof} As $S$ is Gorenstein, so $\syz \cm(S)=\cm(S)$, so condition (1) of Theorem \ref{main11} is equivalent to $\cm(S) \subseteq \syz \cm(R)$. Now we are done by Theorem \ref{main11} and \ref{mat}.  
\end{proof}

\begin{rem} When $R$ admits a canonical module, Corollary \ref{corm} is a special case of \cite[Theorem 5.5]{dms}. 
\end{rem}

\begin{cor}  Let $R$ be a Gorenstein local ring of dimension $1$.  Let $R\subseteq S$ be a finite birational extension. Then, the following are equivalent  

\begin{enumerate}[\rm(1)]

\item $S$ is Gorenstein. 

\item $S\cong \Hom_R(S,R)$. 

\item  $\c_R(S)$ is a stable ideal. 

%\item $\syz \cm(S)=\{M\in \refl(R): \tr_{R}(M)\subseteq \c_R(S)\}$.   
\end{enumerate}
\end{cor}

\begin{proof} As $R$ is Gorenstein, so $\syz \cm(R)=\cm(R)$. Hence, condition (1) of Theorem \ref{main11} is equivalent to $\syz \cm(S)=\cm(S)\cap \cm(R)=\cm(S)$, i.e, $S$ is Gorenstein.  Also, $S\in \cm(S)\subseteq \cm(R)=\refl(R)$ is automatically true. Now we are done by Theorem \ref{main11}.    
\end{proof}

Let $R$ be a generically Gorenstein ring satisfying Serre's condition $(S_1)$. \cite[ Lemma 2.6(1), Proposition 2.9]{gik} establishes a one-to-one correspondence between reflexive finite birational extensions of $R$ and regular reflexive trace ideals of $R$ by sending each such birational extension $S$ to $\c_R(S)$ and each such ideal $I$ to $(I:I)$ (also see \cite[Lemma 2.8]{dms}). Motivated by Theorem \ref{main11} we remark in the following result that when one restricts this map to finite birational extensions with $S\cong S^*$, one gets stable trace ideals on the ideal side. 

\begin{prop} Let $R$ be a ring. If $S$ is a finite birational extension of $R$ such that $S\cong \Hom_R(S,R)$, then $\c_R(S)$ is a regular stable trace ideal. Conversely, if $I$ is a regular stable trace ideal, then $(I:I)$ is a self $R$-dual finite birational extension. 
\end{prop}

\begin{proof} $\c_R(S)$ is a regular ideal as $S$ is module finite, and it is a trace ideal since $\c_R(S)=(R:S)S=\tr_R(S)$. If  $S\cong \Hom_R(S,R)$, then $\c_R(S)\cong S$, so $\Hom_R(\c_R(S),\c_R(S))\cong \Hom_R(S,S)\cong S\cong \c_R(S)$. Thus, $\c_R(S)$ is stable ideal.  

Conversely, if $I$ is a regular stable trace ideal, then $(R:(I:I))\cong (R:I)\cong (I:I)$, where the first isomorphism holds since $I$ is stable, and the second isomorphism holds by \cite[Proposition 2.4(3)]{tk}.  

That the maps are inverses of each other was already observed in \cite[Lemma 2.8]{dms}.  
\end{proof}

\begin{rem}  One nice class of examples of stable trace ideals over $1$-dimension Cohen--Macaulay local rings $R$ are non-parameter Ulrich ideals as introduced in \cite[Definition 2.1]{ulr}. Indeed, if $R$ is $1$-dimensional Cohen--Macaulay, then Ulrich ideals are stable \cite[Proposition 3.2]{dl}.  Now let $I$ be an Ulrich ideal which is not a parameter ideal. Then, $(xR:_RI)=I$ for some non-zero-divisor $x\in I$ such that $I^2=xI$ (\cite[Corollary 2.6(a)]{ulr}). As $x\in I$,so $(xR:_RI)=(xR:I)=x(R:I)$. Thus, $I=x(R:I)$, so $xI=I^2=xI(R:I)$, hence $I=I(R:I)=\tr_R(I)$.       
\end{rem}

\end{document}